\documentclass[11pt,reqno,english]{smfart}
\usepackage{amsgen, amsmath, amsfonts, amsthm, amssymb,enumerate,amscd,graphics,dsfont,comment,tikz-cd,tikz}

\usetikzlibrary{intersections}

\newtheorem{defn}{Definition}[section]
\newtheorem{prop}[defn]{Proposition}
\newtheorem{lem}[defn]{Lemma}
\newtheorem{thm}[defn]{Theorem}
\newtheorem{cor}[defn]{Corollary}
\newtheorem{rem}[defn]{Remark}

\newcommand {\ZZ}{{\mathds Z}}

\newcommand {\XX}{{\mathcal X}}

\newcommand {\CC}{{\mathcal C}}
\newcommand {\Q}{{\mathds Q}}

\newcommand {\UH}{{\mathcal H}}

\newcommand {\M}{{\mathcal M}}

\newcommand {\CP}{{\mathds P}}

\newcommand {\RE}{{{\mathrm E}}}

\newcommand {\TD}{{\tilde{D}}}
\newcommand {\TZ}{{\tilde{Z}}}

\newcommand {\TY}{{\tilde{Y}}}

\newcommand {\CPP}{\CP^1 \times \CP^1}

\newcommand {\TQ}{{\tilde{Q}}}

\newcommand{\tiota}{\tilde{\iota}}
\def\div{\operatorname{div}}
\def\deg{\operatorname{deg}}

\def\Spec{\operatorname{Spec}}

\def\ord{\operatorname{ord}}

\title{Motivic cycles on K3 double covers of del Pezzo surfaces}
\author{Ramesh Sreekantan}
\email{rameshsreekantan@gmail.com/rsreekantan@isibang.ac.in}
\address{Indian Statistical Institute, Bangalore}


\begin{document}
	\baselineskip=17pt
	
	\maketitle

	\begin{abstract}

	In this paper we construct indecomposable motivic cycles in the group $H^3_{\M}(\TZ,\Q(2)$ where $\TZ$ is a $K3$ surface obtained as a branched  double cover of a del Pezzo surface. The cycles are constructed using $(-1)$ curves on the del Pezzo and is a generalization of a recent preprint of Sato \cite{satodelpezzo} of the case when the $K3$ is obtained as the fourfold cover of $\CP^2$ branched at a quartic curve.

		\end{abstract}
	\tableofcontents

	\section{Introduction}
	
	In this  paper we  construct  indecomposable motivic cycles in the group $H^3_{\M}(\TZ,\Q(2))$, where $\TZ$ is a  K3 double cover of a del Pezzo surface,  using $(-1)$ curves. This paper was inspired by the preprint of Sato \cite{satodelpezzo} where he constructs indecomposable motivic cycles in the generic $K3$ surface obtained as a four fold cover of $\CP^2$ ramified at a smooth quartic curve using the $28$ bi-tangents.  This turns out to be a special case of our construction as the double cover of $\CP^2$ ramified at a quartic is a del Pezzo surface of degree $2$ and the $28$ bi-tangents lift to give the $56$ $(-1)$ curves.
	
	In a related paper \cite{sreeK3} we constructed cycles in degree $2$ $K3$ surfaces - which are those obtained as a minimal resolution of a double cover of $\CP^2$ branched at a sextic. In order for those cycles to exist we needed an existence statement from the enumerative geometry of rational curves. The case here is related as the $K3$ double covers of the del Pezzo can be realised as degree $2$ $K3s$. Here the enumerative geometry condition is satisfied by the existence of $(-1)$ curves. At the end of the paper we discuss the more general case of rational curves on del Pezzo surfaces where some results are known thanks to the work of Itzykson and G\"ottsche-Pandharipande \cite{itzy,gopa}.

	\section *{Acknowledgements.} 
	
	I would like to thank Ken Sato for sending me his paper and sharing his ideas about four fold covers of $\CP^2$ branched at  quartics which started me thinking about this question. I would like to thank Microsoft Co-Pilot for helping me find relevant literature and  helping me with the LaTex code for the pictures. I would like to thank the Indian Statistical Institute for their support.

	\section{del Pezzo surfaces}

	All the material in this section can be found in the wonderful book of Dolgachev \cite{dolg}. We found the paper of Comparin, Montero, Prieto–Monta\~{n}ez and Troncoso \cite{CMPT} to be very useful as well.

	Let $X$ be a smooth  algebraic surface over a field $K$ and $K_X$ its canonical bundle. $(X,K_X)$ is called a {\bf del Pezzo  surface} if $-K_X$ is ample. The {\bf degree} of the del Pezzo is $d=K_X^2$. It is known that $1 \leq d \leq 9$.   
	
	There is nice classification of smooth del Pezzo surfaces (\cite{dolg} Proposition 8.1.25.)
	
	\begin{thm} Let $X=X_d$ be a smooth del Pezzo surface of degree $d$. Then 
		\begin{itemize}
			\item If $d=9$, then $X=\CP^2$
			\item If $d=8$, then $X=\CPP$ or the Hirzebruch surface $F_1$, which is $\CP^2$ blown up at a point. 
			\item If $1 \leq d \leq 7$, $X=Bl_{P_1,\dots,P_r}(\CP^2)$, the blow up of $\CP^2$ at $r=9-d$ points in general position. 
		\end{itemize}
		
		Here {\em general position} means no three points lie on a line, no six points lie on a conic and no 8 points lie on a nodal cubic with one of them being the node.
		
	\end{thm}
	
	Many classical surfaces are  del Pezzo surfaces. Apart from $\CP^2$, $\CPP$ and $F_1$, cubic hypersurfaces in $\CP^3$ are del Pezzo surfaces of degree 3.

	\subsection{$(-1)$-curves on a del Pezzo surface}

	A $(-1)$ curve is a curve $D$  on $X$ such that $D^2=-1$. Any curve of negative self intersection on a smooth del Pezzo surface is a $(-1)$ curve.  It appears that there are only finitely many of them -- the number depends only on the degree. 
	
	\begin{thm}
		Let $X=Bl_{P_1,\dots,P_r}(\CP^2)$ be a del Pezzo of degree $1 \leq d \leq 8$ which is a blow up of $\CP^2$ at $r=9-d$ points in general position. Let $\pi:X \rightarrow \CP^2$ be the map induced by the blow ups. Then if $D$ is a $(-1)$ curve, the image $\pi(D)$ is one of the following:
		
		\begin{itemize}
			\item One of the points $P_i$ 
			\item A line $L_{ij}$ passing through two of the points $P_i$ and $P_j$. 
			\item A conic $F_{i_1,\dots,i_5}$ passing through five of the points $P_{i_1},\dots,P_{i_5}$. 
			\item A cubic $C_{3,i}(j)$ passing through $7$ of the points except $P_i$ with a node at $P_j$.  
				\item A quartic $Q_4(ijk)$ passing through all $8$ of the points with nodes at $P_i, P_j$ and $P_j$.
			  	\item A quintic $Q_5(ijklmn)$ passing through all the $8$ points with double points at the $6$ points  $P_i,\dots,P_n$. 
			  	\item A sextic 	$S_6(j)$ passing through all the $8$ points with double points at $7$ of them and a triple point at $P_j$. 
			  \end{itemize}

			  The number $n$ of $(-1)$ curves on $X$ is determined by the degree $d$. One has the following table 
			  
			  $$\begin{tabular}{|c||c|c|c|c|c|c|c|c|c|}
			  	\hline
			  	d & 9 & 8 & 7 & 6 & 5 & 4 & 3 &2 & 1 \\
			  	\hline
			  	n & 0 & 1 & 3 & 6 & 10 & 16 & 27 & 56 & 240  \\
			  	\hline
			  \end{tabular}$$
			  
			  \label{minusonecurve}
			  
			  \end{thm}
			  
			  \subsection{K3 double covers of  del Pezzo surfaces}
			  
			  If $X$ is $\CP^2$ then it is well known that a double cover of $\CP^2$ branched at a smooth sextic $S$ is a $K3$ surface. These are called {\em degree 2} $K3$ surfaces and their moduli was studied in great detail by Shah \cite{shah}. More generally, if $S$ has nodes, the minimal desingularization of the branched double cover is a $K3$ surface. For smooth sextics, outside a countable union of closed subvarieties, the $K3$ surface has Picard number $1$. For the nodal ones, the generic  Picard number is determined by the nodes, as the exceptional divisors are independent elements of the Neron-Severi group. 
			  
			  For example, the Kummer surface of a simple Abelian surface is a special case of this when the sextic is a product of six lines tangent to a conic. Here the generic Picard number is $17$.

			 This construction generalizes to del Pezzo surfaces. A theorem of Alexeev and Nikulin  \cite{alni} states that if $X$ is a del Pezzo of degree $d$ then there is smooth curve $C$ in $|-2K_X|$ such that the double cover branched at $C$ is a $K3$ surface. 
			 
			 In fact it can be seen that if $X$ is a variety with ample anti-canonical bundle then the double cover branched at a divisor in $|-2K_X|$ has trivial canonical bundle. 
			 
			 Let $\TZ \stackrel{\phi}{\longrightarrow} X_d$ be the $K3$ double cover branched at a curve $C$ in $|-2K_X|$. One has 
			 $$-2K_X=6H- \sum_{i=1}^{9-d} 2 E_{P_i}$$
			 in the Picard group of $X$, where $H$ is the pull-back of the hyperplane section under the map to $\CP^2$. Hence $C$ can be considered as the strict transform of a singular sextic $S$ in $\CP^2$ with nodes at the points $P_i$. Therefore $\TZ$ can be thought of as the minimal desingularization of the double cover of $\CP^2$ branched at $S$, which is a $K3$ surface of degree $2$. Let $Z$ be the double cover of $\CP^2$ branched at $S$. This has nodes at the points lying over $P_i$. 
			 
			 We have the following picture:
			 
			\begin{center}
			  
			  \begin{tikzcd}
			  	
			  	\TZ \arrow[d,"\phi_C"] \arrow[r,dashrightarrow,"\tilde{\pi}"]& Z \arrow[d,"\psi_C"]\\
			  	X_d \arrow[r,dashrightarrow,"\pi"] & \CP^2 
			  	
			   \end{tikzcd}
			  
			  \end{center}

			   \section{Motivic cycles} 
			   \subsection{Presentation of cycles in $H^3_{\M}(X,\Q(2))$.}
			   
			   Let $X$ be a surface over a field $K$. The group $H^3_{\M}(X,\Q(2))$ has the following presentation. It is generated by sums of the form 
			   $$\sum (C_i,f_i)$$
			   where $C_i$ are curves on $X$ and $f_i$ functions on them satisfying the co-cycle condition 
			   $$\sum \div(f_i)=0$$
			   Equivalently these can be though of as codimensional  $2$ subvarieties $Z$ of $X \times \CP^1$ such that 
			   $$Z\cdot (X \times \{0\} -X \times \{\infty\})=0$$
			   The cycle given by the sum of the graphs of $f_i$ in $X \times \CP^1$ has this property. 
			   
			   Relations are given by the Tame symbol of a pair of functions. If $f$ and $g$ are two functions on $X$, the  Tame symbol of the pair
			   $$\tau(f,g)=\sum_{Z \in X^1} \left(Z,(-1)^{\ord_g(Z)\ord_f(Z)}\frac{f^{\ord_g(Z)}}{g^{\ord_f(Z)}}\right)$$
			   satisfies the co-cycle condition and  such cycles are  defined to be $0$ is the group. Here $X^1$ is the set of irreducible codimensional one subvarieties. 
			   
			   One example of an element of this group is a cycle of the form $(C,a)$ where $a$ is a nonzero constant function and $C$ is a curve on $X$. This trivially satisfies the cocycle condition. More generally, if $L/K$ is a finite extension and $X_L=X \times_{\Spec(K)} \Spec(L)$ then one has a product map 
			   $$\bigoplus_{L/K} H^2_{\M}(X_L,\Q(1)) \otimes H^1_{\M}(X_L,\Q(1)) \longrightarrow \bigoplus_{L/K} H^3_{\M}(X_L,\Q(2)) \stackrel{Nm_{L:K}}{\longrightarrow} H^3_{\M}(X,\Q(2))$$
			   The image of this is the group of {\em decomposable cyles} $H^3_{\M}(X,\Q(2))_{dec}$. The group of {\em indecomposable cycles} is the quotient group
			   $$H^3_{\M}(X,\Q(2))_{indec}=H^3_{\M}(X,\Q(2))/H^3_{\M}(X,\Q(2))_{dec}$$
			   In general it is not clear how to construct non-trivial elements of this group.  
			   
			   A way of constructing cycles in this group is as follows. Since we use it in the main construction we label it a proposition.

			   \begin{prop}  Let $Q_1$ and $Q_2$  be two  {\em  rational curves} on $X$ which meet at two points $P_1$ and $P_2$. Let $f_{Q_1}$ be a function on $Q_1$ with $\div_{Q_1}(f_{Q_1})=P_1-P_2$ and $f_{Q_2}$ be a function on $Q_2$  with $\div_{Q_2}(f_{Q_2})=P_2-P_1$. Then 
			   	$$(Q_1,f_{Q_1})+(Q_2,f_{Q_2})$$
			   	satisfy the co-cycle condition and hence determine an element of $H^3_{\M}(X,\Q(2))$ 
			   	
			   	\label{construction}
			   	\end{prop}

			   
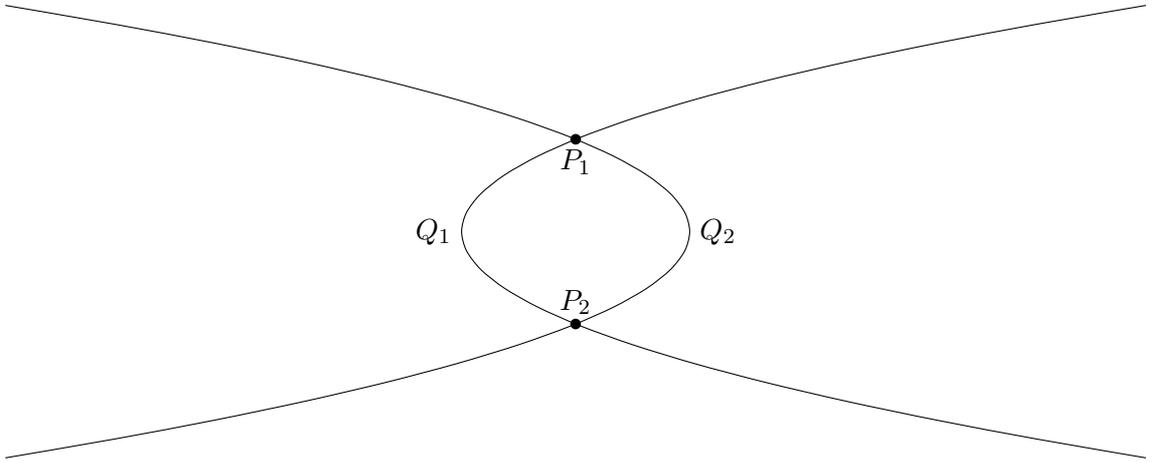
\vspace{\baselineskip}

			   \begin{center}
			   	
			   	\begin{tikzpicture}
			   		
			   		\draw[name path=conic1,domain=-3:3,smooth,variable=\x,black] plot ({\x*\x},{\x});
			   		
			   		\node[left] at  (0,0) {$Q_1$};
			   		

			   		\draw[name path=conic2,domain=-3:3,smooth,variable=\x,black] plot ({3-\x*\x},{\x});
			   		
			   		\path[name intersections={of=conic1 and conic2,by={A1,A2}}];
			   		
			   		\node [right] at (3,0) {$Q_2$};
			   		\foreach \point in {A1,A2}
			   		\fill (\point) circle (2pt);
			   		\node[above ] at (A1) {$P_2$};
			   		\node[below ] at (A2) {$P_1$};
			   		
			   	\end{tikzpicture}

			   \end{center}
			   \vspace{\baselineskip}
			   
			    A notable special case is when $X$ is the desingularization of a surface $Y$ with isolated nodal points, and $Q$ is a curve which has a node at one of the nodes of the surface. Let $\nu:\tilde{Q}\rightarrow Q$ be the strict transform in the blow up $X$ of $Y$ at $P$.  $\tilde{Q}$ meets the exceptional fibre $\RE_P$ at two points $P_1$ and $P_2$. Both $\tilde{Q}$ and $\RE_P$ are rational curves. Let $f_P$ be the function with $\div(f_P)=P_1-P_2$ on $\tilde{Q}$ and similarly let $g_P$ be the function with $\div(g_P)=P_2-P_1$ on $\RE_P$. Then 
			   $$(\tilde{Q},f_P)+(\RE_P,g_P)$$
			   is an element of $H^3_{\M}(X,\Q(2))$.  

			   This construction was used by Lewis and Chen \cite{chle} to prove the Hodge ${\mathcal D}$-conjecture for certain $K3$ surfaces  and by me \cite{sree2014} to prove an analogue of that in the non-Archimedean case for Abelian surfaces.

			   A priori  it is not clear that if these  cycles are indecomposable or even non-trivial.  We will show that in the case of the cycles we construct from $(-1)$-curves on del Pezzo surfaces it is the case `in general'. For this we need to use the localization sequence in motivic cohomology.

			   \subsection{The localization sequence}
			   \label{localization}
			   
			   If $\XX \rightarrow S$ is a family of varieties over a base $S$ with $X_{\eta}$ the generic fibre and $X_s$ the fibre over a closed subvariety $s$  then one has a localization sequence relating the three. 
			   $$ \cdots \rightarrow H^3_{\M}(X_{\eta},\Q(2) \stackrel{\partial}{\longrightarrow} \bigoplus_{s \in S^1} H^2_{\M}(X_s,\Q(1)) \longrightarrow H^4_{\M}(\XX,\Q(2)) \rightarrow \cdots $$
			   where $S^1$ denotes the set of irreducible closed  subvarieties of codimension $1$. The boundary map of an element 
			   $Z=\sum_i (C_{\eta,i},f_{\eta,i})$ is given as follows. Let $\CC$ denote the closure of $C_{\eta}$ in $\XX$. Then 
			   $$\partial(Z)=\sum_i \div_{\CC_i}(f_{\eta,i})$$
			   The cocycle condition $\sum_i \div_{C_{\eta,i}}(f_{\eta,i})=0$ implies that the `horizontal' divisors cancel out and only the components in some of the  fibres survive. 
			   
			   Let $\XX$ be a family of $K3$ surfaces. To check for indecomposability one can use the boundary map $\partial$. Over the algebraic closure of the base field, a decomposable cycle in $H^3_{\M}(X_{\eta},\Q(2))$ is given by $(C_{\eta},a_{\eta})$ where $C_{\eta}$ is a curve on $X_{\eta}$ and $a_{\eta}$ a constant (that is, a function on the base $S$). 
			   
			   The boundary map is particularly simple to compute in the case of a decomposable element 
			   $$\partial((C_{\eta},a_{\eta}))=\sum \ord_s(a_{\eta}) C_s$$
			   where $C_s$ is the restriction of the closure $\CC$ of $C_{\eta}$ to the fibre over $s$. Hence in the boundary of decomposable elements one can obtain only those cycles in the special fibres which are the restrictions of cycles in the generic fibre. In particular, if $Z_{\eta}$ is a motivic cycle and the boundary contains cycles which are {\em not} the restriction of cycles in the generic fibre then the motivic cycle is necessarily indecomposable. We will show that that is the case for our cycles.
			   
			   An alternate way, as is done in Chen-Lewis \cite{chle}, Ma-Sato \cite{masa} or Sato \cite{sato}, is to compute the Beilinson regulator and show that, when evaluated against a particular transcendental $(1,1)$ form, this is non-zero.  In the case when the base is the ring of integers of a glabal field or a local field, the boundary map can be thought of as a non-Archimedean regulator map and the idea is essentially the same. The transcendental form is represented by a new algebraic cycle in the special fibre.  Evaluation against this form corresponds to intersecting with the new cycle.  That being non-zero implies that the component of the boundary in the direction of the new cycle is non-zero. This also fits in to the philosophy espoused by Manin \cite{mani} that the Archimedean component of an arithmetic variety can be viewed as the `special fibre at $\infty$'.

			   \section{Motivic cycles from $(-1)$-curves}

			   We now return to the special case of $K3$ double covers of del Pezzo surfaces. Recall that there are finitely many $(-1)$ curves on a del Pezzo surface. These are all rational curves. We would like to use Proposition \ref{construction} to obtain cycles on the $K3$ double cover $\TZ$.
			   
			   \subsection{Construction}
			   
			   Let $D$ be a $-1$-curve. From Theorem \ref{minusonecurve} we know  the image of $D$ in $\CP^2$ under $\pi$. 
			   
			   \begin{lem} Let $Q$ be the image of a $(-1)$-curve $D$ under the map $\pi$. Assume $Q$ is not a point. Let $e$ be its degree. The $Q$ meets the branched sextic $S$ at $3e-1$ double points and two other points $t_Q$ and $s_Q$. 
			   \end{lem}
			   \begin{proof} This is essentially a case by case observation. If $Q=L_{ij}$ then $\deg(Q)=1$ and $(Q,S)=6$. Since it passes through the points $P_i$ and $P_j$ which are of multiplicity $2$, they account for $4$ of the points of intersection. Hence they have to be two more. 
			   	
			   	Similarly, for the conic $F_{i_1,\dots,i_5}$ there are $12$ points of intersection with $S$, $10$ of which are accounted for by the $P_{i_j}$. All the other cases are similar. For example, for $S_6(j)$ is of degree $6$ so should meet $S$ at $36$ points. Since it has  double points at $7$ points $P_i$, for $i \neq j$, this accounts for $28$ points. The triple point at $P_j$ accounts for another $6$ more points hence once again $S_6(j)$ and $S$ meet at $2$ other points.

			   \end{proof}

			  Alternately we can work directly on $X_d$. $\TZ$ is branched at a curve $C$ in $|-2K|$. The $(-1)$ curves in the Picard group are of the form 
			  \begin{itemize}
			  	\item  $(0,-1,0.0,0,0,0,0,0)$ - exceptional curves  
			  	\item  $(1,1,1,0,0,0,0,0,0)$ - lines through $2$ points $P_i$ 
			  	\item  $(2, 1,1,1,1,1,0,0,0)$ - conics through $5$ points. 
			  	\item $ (3,2,1,1,1,1,1,1,0)$ -  cubics through $7$ points with a double point at one of them. 
			  	\item $(4,2,2,2,1,1,1,1,1)$ - quartics with $3$ double points passing through all $8$ points. 
			   \item $(5,2,2,2,2,2,2,1,1)$ - quintics through $8$ points with $6$ double points 
			   \item $(6,3,2,2,2,2,2,2,2)$ - sextics with $7$ double points and one triple point. 
			   \end{itemize}
			   
			   Recall that $K=(-3,1,\dots,1)$. Since $C \in |-2K|$ a simple calculation shows that if $D$ is a $(-1)$-curve
			   $$(C.D)=2$$
			   These are the two points  $t_D$ and $s_D$ lying over $t_Q$ and $s_Q$.

			   If we look at the images $Q$ of the curves $D$ in $\CP^2$, we are now precisely in the situation of \cite{sreeK3} where we constructed motivic cycles on degree $2$ $K3$ surfaces subject to the condition that there existed certain rational curves. The images $Q$ of the $(-1)$ satisfy this condition. 
			   
			   We recall the construction here - keeping in mind that we are working with the $(-1)$ curves $D$ on $X$ and not on $\CP^2$. 
			   
			   \begin{thm}
			   	Let $\TZ$ be the $K3$ double cover of a del Pezzo surface $X_d$ of degree $d$.

			   Let $D_1$ and $D_2$ be two $(-1)$ curves lying over the curves $Q_1$ and $Q_2$ in $\CP^2$. Assume $Q_1$ and $Q_2$ meet at a point $P$ and the points $s_{Q_1},s_{Q_2},t_{Q_1}$ and $t_{Q_2}$ are distinct.  Then $Q_1,Q_2$ and $P$ determine a  motivic cycle $\Xi_{Q_1,Q_2,P}$ in the group $H^3_{\M}(\TZ,\Q(2))$.  
			   	\label{motiviccycle}
			   \end{thm}

			  \begin{proof}
			  	
			  	There are two sub-cases. 
			  	
			  	\noindent{\bf  Case 1: $P$ lies on the branch locus}. In this case first  we assume that one of $Q_1$ or $Q_2$ is the exceptional fibre at $P$, $\RE_P$, say $Q_2=\RE_P$. In that case the $(-1)$ curves $D_1$ and $D_2$ meet at $1$ point lying over $P$ and in the double cover $\TZ$ there are two points of intersection $P_1$ and $P_2$ of the curves $\TD_1$ and $\TD_2$ lying over $D_1$ and $D_2$ in $\TZ$. 
			  	
			  	If neither $Q_1$ nor $Q_2$ is an exceptional fibre, let $Q$ be the exceptional fibre at $P$. Let $\Xi_{Q_1,Q,P}$ and $\Xi_{Q,Q_2,P}$ be as defined below and define 
			  	$$\Xi_{Q_1,Q_2,P}=\Xi_{Q_1,Q,P}-\Xi_{Q_1,Q_2,P}$$

			  	\noindent{\bf Case 2: $P$ does not lie on the branch locus.} In this case there are two points $P_1$ and $P_2$ lying over $P$ in $\TZ$ and the curves $\TD_1$ and $\TD_2$ meet at $P_1$ and $P_2$. 
			  	
			  	If $D$ is a $(-1)$ curve, let $s_{\TD}$ and $t_{\TD}$ be the two points of $\TD$ lying over the points $s_{D}$ and $t_{D}$ of $C.D$.  Let $f_{\TD_1}$ be the function on $\TD_1$  with $\div_{\TD_1}(f_{\TD_1})=P_1-P_2$ and $f_{\TD_2}$ be the function on $\TD_2$ with $\div_{\TD_2}(f_{\TD_2})=P_2-P_1$.

			  	We further assume that $f_{\TD_i}(s_{\TD_i})=1$, where $s_{\TD_i}$ is the point on the branch locus lying over $s_{D_i}$.  As shown in Proposition \ref{construction}, one has a cycle 
			  	$$\Xi_{Q_1,Q_2,P}=(\TD_1,f_{\TD_1})+(\TD_2,f_{\TD_2})$$ 
			  	in the motivic cohomology group $H^3_{\M}(\TZ,\Q(2))$. 
			  	
			  	\end{proof}
			  	
			  	We need to argue that there exist points of intersection of $D_i$ and $D_j$ for two $(-1)$ curves $D_1$ and $D_2$. This can be seen by looking at the intersection number of the images in $\CP2$. Since we know the degree there its easy to compute the number of points of intersection and we can see in many instances that there has to be a point of intersection outside the branch locus. For instance, if $D_1$ lies over a conic and $D_2$ lies over a cubic, there have to be six points of intersection but the conic meets the  branch locus at $5$ points $P_1$ so they meet at another point outside the branch locus. 
			  	
			  	The construction here relies on the fact that there exist  rational curves meeting the branch locus at double points except at two points.  And for those surfaces where the rational curve meets the branch locus everywhere at double points the rank of the Neron-Severi of the double cover increases. 
			  	
			  	This is something special to this case. There are other surfaces - like Horikawa surfaces, which appear as double covers of rational surfaces branched at curves. In those cases, however, it is not clear that such rational curves exist. For instance, if the branch locus is an octic then it is not clear that there exists a line meeting the octic at 4 bi-tangents. In this case it appears to be a happy accident of enumerative geometry.

			   As things stand we do not know if the cycles $\Xi_{Q_1,Q_2,P}$ in
			   $H^3_{\M}(\TZ,\Q(2))$ are non-trivial. We will show that in fact they  are generically indecomposable. These cycles are defined as long as the two other points of intersection of the curves with $C$ are distinct.

			   	\subsection{Indecomposability}
			   	
			   	We first discuss what happens when $s_D$ and $t_D$ coincide. 
			   	
			   	\begin{prop}
			   		Let $D_i$ be a $(-1)$ curve on $X_d$. Let $s_{D_i}$ and $t_{D_i}$ be the two points of intersection of $D_i$ with $C$. If $s_{D_i}$ and $t_{D_i}$ coincide then $\phi_C^{-1}(D)$ has two components $\TD_{i,1}$ and $\TD_{i,2}$. Equivalently, if $Q_i$ is the image of $D_i$ in $\CP^2$ which is not a point and $s_{Q_i}$ and $t_{Q_i}$ coincide then $\psi_S^{-1}(Q_i)$ has two components $\TQ_{i,1}$ and $\TQ_{i,2}$. 
			   	\end{prop}	
			   	
			   	\begin{proof} If the two points coincide then all the ramification points of the map $\TQ_i \rightarrow Q_i$ and $\TD_i \rightarrow D_i$ are nodal. Hence the induced map from the normalization of $\TQ_i$ or $\TD_i$ to $\CP^1$ is an {\em unramified} double cover. However, $\CP^1$  does not have an irreducible unramified double cover as one can see from the Riemann-Hurwitz formula or the fact that $\CP^1$ is simply connected. Hence $\TQ_i$ and $\TD_i$ have two components $\TQ_{i,2}$ and $\TQ_{i,2}$ and $\TD_{i,1}$ and $\TD_{i,2}$. The points of intersection of $\TQ_{i,1}$ and $\TQ_{i,2}$ are the nodes of $\TQ_i$. 
			   		
			   		\end{proof} 
			   		
			   The components $\TD_1$ and $\TD_2$ are new elements of the Neron-Severi group of $\TZ$. Hence if $s_Q=t_Q$ the Picard number of  $\TZ$ increases by $1$. Note that the involution induced by the double cover interchanges $\TD_1$ and $\TD_2$. 
			   
			    The moduli of  $K3$ surfaces where the rank of the Neron-Severi is one more than the generic rank is usually called a {\em Heegner divisor} or a {\em Noether-Lefschetz} locus. In the special case of Kummer surfaces of Abelian surfaces these are also called {\em Humbert Surfaces} and correspond to when the Abelian surface has multiplication by a real quadratic field. Let $H_Q$ denote the submoduli where there is a rational curve $Q$ corresponding to a $(-1)$ curve which meets the nodal sextic only at double points - that is, the points $s_Q$ and $t_Q$ coincide. 
			   	
			    The construction we made of a cycle in $H^3_{\M}(\TZ,\Q(2))$ was for a fixed $K3$ surface. However, if $M_d$ is the moduli of del Pezzo surfaces of degree $d$, the points $P_1$ and $P_2$ are the intersections of the pre-images of two $(-1)$ curves, so are defined everywhere and determine a section of the universal family over the moduli space. In particular, we have an algebraic family of cycles $\Xi_{Q_{1,z},Q_{2,z},P_z}$ in $H^3_{\M}(\TZ_z,\Q(2))$ as long as the points $s_{Q_{i,z}}$ and $t_{Q_{i,z}}$ for $i \in \{1,2\}$ are distinct. This defines a cycle $\Xi_{Q_{1,U},Q_{2,U},P_{U}}$ defined in the fibre over the open set $U=M_d\backslash \{H_{Q_1} \cup H_{Q_2}\}$.

			    In the fibres over $H_Q$, where $s_Q$ and $t_Q$ coincide, the points $P_1$ and $P_2$ are interchanged by the involution. Hence without loss of generality we may assume that $P_1$ lies on $\TD_1$ and $P_2$ lies on $\TD_2$.
			    
			    We then have the following theorem:

			   \begin{thm} Let $H_{Q_1}$ be the submoduli of moduli space of del Pezzo surfaces of degree $d$ such that $s_{Q_1}$ and $t_{Q_1}$ coincide. Let $\Xi_{Q_{1,U},Q_{2,U},P_U}$ be the cycle in $H^3_{\M}(\TZ_U,\Q(2))$ constructed in Theorem \ref{motiviccycle}, where $U=M_d \backslash \{H_{Q_1} \cup H_{Q_2}\}$. We assume that the section $P_i$ intersects $\TD_i$.  Then, under the boundary map of the localization sequence, 
			   	$$\partial (\Xi_{Q_{1,U},Q_{2,U},P_{U}})=\TD_{1,1}-\TD_{1,2}$$

	\end{thm}
	
	\begin{proof} Let
		 $$\Xi_{Q_{1,U},Q_{2,U},P_U}=(\TD_{1,U},f_{\TD_1,U}+(\TD_2,f_{\TD_2,U})$$ 
		 be the cycle with 
		 $$\div(f_{\TD_1,U})=P_{U,1}-P_{U,2}$$
		 and 
		 $$\div(f_{\TD_2,U})=P_{U,2}-P_{U,1}$$
		 Recall that we have the following diagram
		 
		 \begin{center}
		 	
		 	\begin{tikzcd}
		 		
		 		\TZ \arrow[d,"\phi_C"] \arrow[r,dashrightarrow,"\tilde{\pi}"]& Z \arrow[d,"\psi_C"]\\
		 		X_d \arrow[r,dashrightarrow,"\pi"] & \CP^2 
		 		
		 	\end{tikzcd}
		 	
		 \end{center}

		 Let  $\iota$ be the involution on $Z_U$ induced by the double cover $\psi_S:Z \rightarrow \CP^2$ and let $\tiota$ be the involution induced by the double cover $\TZ \rightarrow X_d$. The fixed points of $\iota|_{\TQ}$ are the ramification points of the double cover to $Q$.  
		 The involution $\iota$  lifts to the involution $\tiota$ on $\TD$ and 
		 $$ \tiota(P_{U,1})=P_{U,2}$$
		 Since $\TD_{i,U}$ is the pre-image of $D_{i,U}$, $\tiota$ stabilizes it: $\tiota(\TD_i)=\TD_i$. 
		 Since $s_{Q_{i,U}}$ and $t_{Q_{i,U}}$ are on the branch locus we  have $\tiota(s_{Q_{i,U}})=s_{Q_{i,U}}$ and $\tiota(t_{Q_{i,U}})=t_{Q_{i,U}}$. 
		 
		 Let $f=f_{\TD_1}$ and let $f^{\tiota}=f\circ \tiota$. Then 
		 $$\div(f^{\tiota})=P_{U,2}-P_{U,1}$$
		 Hence $f^{\tiota}=\frac{c_{U}}{f}$, where $c_U$ is a `constant' on $U$. Since we have assumed $f(s_{\TD_{1,U}})=1$ and $f^{\tiota}(s_{Q_U})=f(\tiota(s_{Q_U}))=f(s_Q)=1$, we have $c_U \equiv 1$.
		 
		 To compute the boundary of $\Xi_{Q_{1,U},Q_{2,U},P_{U}}$, we need to compute the divisor of $f_{U}$ on the closure of $\TD_{1,U}$. Since $\TD_1$ gains a node over $H_Q$ we have to consider the divisor on the blow up of $\TD_1$ at the node $R_Q$.

		 Over $H_{Q_1}-H_{Q_2}$, the closure has three components, $\TD_{1,1,H_Q}$, $\TD_{1,2,H_Q}$ and the exceptional fibre $\RE_{R_Q}$. The involution $\tiota$ interchanges them and as remarked above, we may assume that the closure of $P_{i,U}$ lies on $\TD_{1,i,H_Q}$. Further, since $\RE_{R_Q}$ is the exceptional fibre over the node $R_Q=\TD_{1,1,H_Q} \cap \TD_{1,2,H_Q}$, the involution stabilizes $\RE_{R_Q}:\tiota (\RE_{R_Q})=\RE_{R_Q})$. 
		 
		 On the closure of $\TD_{1,U}$, 
		 $$\div_{\bar{\TD}_{1,U}}(f_U)=\UH+ a\TD_{1,1,H_Q}+b \TY_{1,2,H_Q}+ c\RE_{R_Q}$$
		 for some $a, b$ and $c$ in $\ZZ$, where $\UH=\overline{\div_{Y_U}(f_U)}$.	is the closure of the horizontal divisor.
		 
		 We recall that $\div(f^{\tiota})=\tiota(\div(f))$ and $f^{\tiota}=1/f$, so, on one hand 
		 	 $$\div_{\bar{\TD}_{1,U}}(f^{\tiota}_U)=\tiota(\UH+ a\TD_{1,1,H_Q}+b \TD_{1,2,H_Q} + c \RE_{R_Q})=-\UH+a\TD_{1,2,H_Q}+b \TD_{1,1,H_Q} + c\RE_{R_Q}$$
		 on the other hand 
		 $$\div_{\bar{\TD}_{1,U}}(f^{\tiota}_U)=-\div_{\bar{\TD}_{1,U}}(f_U)=-\UH-a\TD_{1,1,H_Q}-b \TD_{1,2,H_Q}-c\RE_{R_Q}$$
		 Comparing coefficients we have $a=-b$ and $c=0$. Hence 
		 $$\div_{\bar{\TD}_{1,U}}(f_U)=\UH+ a(\TD_{1,1,H_{Q_1}}- \TD_{1,2,H_{Q_1}})$$
		 Now we consider $\TD_{2,U}$. This remains irreducible over $H_{Q_1}-H_{Q_2}$. Hence 
		 $$\div_{\overline{\TD_{2,U}}}(f_{\TD_{2,U}})=-\UH$$
		 Adding this to the above result shows that 
		 $$\partial (\Xi_{Q_U,P_U})=a(\TD_{1,1}-\TD_{1,2})$$

		 We now show $a=1$. The moduli space of del Pezzo surfaces of degree $d$ has trivial rational $H^1$, so homological and rational equivalence are the same. Therefore we can find a function $h$ such that 
		 $$\div(h)=H_{Q_1} - \sum a_Z Z $$
		 for some divisors $Z$ which we may assume are disjoint from $H_{Q_1}$. Let $(\TD_U.h)$ be the decomposable cycle in $H^3_{\M}(\TZ_U,\Q(2))$. The boundary of this cycle is 
		 $$ \partial((\TY,h))=\TD_{1,1,H_{Q_1}}+\TD_{1,2,H_{Q_1}}+ \RE_{R_Q} - \sum a_Z \TD_{1,Z}$$
		 where $\TD_{1,Z}$ is the restriction of $\TD_1$ to the fibres over $Z$. The $\TD_{1,Z}$ do not intersect $\TD_{1,i,H_Q}$ or $\RE_{R_Q}$.  
		 
		 Consider the function $f_U h^{-a}$ . This has divisor  
		 $$\div(f_U h^{-a})= \UH+ a(\TD_{1,1,H_{Q_1}}-\TD_{1,2,H_{Q_1}})-a(\TD_{1,1,H_{Q_1}}+\TD_{1,2,H_{Q_1}}+\RE_{R_Q}) + a(\sum a_Z \TD_{1,Z})$$
		 $$=\UH-2a \TD_{1,2} - a\RE_{R_Q} +  a(\sum a_Z \TD_{1,Z})$$
		 In particular, the support of the divisor does not contain $\TD_{1,1,H_{Q_1}}$. Hence $f_Uh^{-a}|_{\TY_{1,H_{Q_1}}}$ is a function on $\TD_{1,1,H_{Q_1}}$ and one has 
		 $$\deg(\div(f_U h^{-a}|_{\TD_{1,1,H_{Q_1}}}))=0$$
		 On the other hand, 
		 $$\deg(\div(f_U h^{-a}|_{\TD_{1,1,H_{Q_1}}}))=(\div(f_U h^{-a}),\TD_{1,1,H_{Q_1}})$$
		 which implies 
		 $$(\UH,\TD_{1,1,H_{Q_1}})-2a(\TD_{1,2,H_{Q_1}},\TD_{1,H_{Q_1}}) -a(\RE_{R_Q},\TD_{1,1,H_Q} + a (\sum a_Z \TD_Z, \TD_{1,1,H_{Q_1}})=0$$
		 We know 
		 $$(\UH,\TD_{1,1,H_{Q_1}})=1 \;\; (\TD_{1,2,H_{Q_1}},\TD_{1,1,H_Q})=0 \;\; (\RE_{R_Q},\TD_{1,1,H_{Q_1}})=1$$ 
		 and 
		 $$(\TD_Z,\TD_{1,1,H_{Q_1}})=0 \text{ for all $Z$}$$
		 This gives $0=1-a$. Hence $a=1$. 
		 		
		\end{proof}
		
		\begin{rem} On the set $H_{Q_1} \cap H_{Q_2}$ the boundary is more complicated and involves both the new cycles $\TD_{1,1}$ and $\TD_{2,1}$ 
			
			\end{rem}
		
		\begin{cor} The cycles $\Xi_{Q_1,Q_2,P}$   are generically indecomposable. 
			\end{cor} 
		\begin{proof}	
			
			The boundary of a decomposable cycle cannot involve cycles which are not the restriction of generic cycles. Since $\TD_{1,1,H_{Q_1}}-\TD_{1,2,H_{Q_1}}$ is not the restriction of a generic cycle, our cycle is not decomposable. 
			\end{proof}
		
		In the case of K3 double covers of degree $2$ del Pezzo surfaces, this is a result of Sato \cite{satodelpezzo}, which is what inspired this paper. However, our methods are different. His proof of generic indecomposability is to show that the Beilinson regular of the family of cycles evaluated on a specific $(1,1)$ form is a is non-trivial function. Our argument can be considered a `non-Archimedean' version of that.

		\section{Degree $d$ del Pezzo surfaces and the degeneracy loci}
		
		In this section we describe the degeneracy loci $H_Q$. These are typically interesting submoduli of the moduli of del Pezzo surfaces which have been classically studied. 
		
		\subsection{Degree 1}
		
		In this case the map induced by $|-2K_X|$ determines a map to $\CP^3$ and the del Pezzo surface $X_1$ can be realized as a double cover of a quadric cone branched at a sextic, which is the complete intersection of a cubic surface with the quadric cone. The 240 $(-1)$ curves are obtained as components of the double cover of a conic cut out by a {\em tri-tangent plane} - which is a plane which meets the branched sextic at three points of tangency. The two $(-1)$ curves lying over the conic meet at three points lying over these points of tangency. 
		
		From van Luijk and Winter \cite{wiva}, Lemma 2.7 one can see that a $(-1)$ curve meets 183 other curves at least at $1$ point and those can be used to construct motivic cycles.

		\subsection{Degree 2}
		This is the case studied by Sato \cite{satodelpezzo}. The $K3$ surface is double cover of the degree $2$ del Pezzo branched $X_2$ at a curve $C$. The curve $C$ itself is the branch locus of the canonical map induced by $-2K$ to $\CP^2$ and is a quartic curve in $\CP^2$. 
		
		It is a well known classical fact that a quartic curve in $\CP^2$ has $28$ bi-tangents. The $56$  $(-1)$ curves in $X_2$ are the curves lying over these bi-tangents. 
		
	    The $(-1)$ curve $D$ we use to construct the cycle lies over a line $Q$ in $\CP^2$ which is a bi-tangent to the quartic curve $C$. In general there will be two points of bi-tangency on the quartic. The two points lying over these bi-tangents are the points $s_Q$ and $t_Q$. 
	    
	    When the two bi-tangents coincide it is called a point of inflection. When this happens it means that the corresponding del Pezzo corresponds to a point of $H_Q$  and the cycle $\Xi_{Q,P}$ is not defined. 
	    
	    For a given bi-tangent $Q$ in $\CP^2$ there are two curves in $X_2$. These two curves $D_1$ and $D_2$ meet at two points. Recall that under the blow down map to $\CP^2$ the $(-1)$ curves map to one of four possibilities - a point - in which case the curve is an exceptional curve, a line through two of the singular points, a conic through five of the singular points or a cubic through seven points, one of which is a double point. A little analysis shows that the two curves lying over a bi-tangent are either a pair of an exceptional curve and a cubic or a line and a conic. Further, these pairs do not meet at any of the exceptional points $P$, they meet at the points lying over the points of bi-tangency. 
	    
	    In his construction Sato uses two bi-tangent lines to build the cycle. In fact, looking at his construction more carefully shows that  he uses two $(-1)$ curves in the del Pezzo which meet at two points which do not come from the same bi-tangent. So one has that the cycles we have are essentially those he constructs. His proof of indecomposability is quite different though - he computes the real regulator and shows it is non-trivial by evaluating it at the $X_2$ which is the double cover of $\CP^2$ branched at the Fermat quartic. 
	    
	    The locus where the cycle degenerates is when the two points $s_{D_1}$ and $t_{D_1}$ coincide. In terms of the quadrics this is when the bi-tangents is an infection bi-tangent - namely those curves where the bi-tangent meets at a point of multiplicity four.

	    \subsection{Degree 3} 
	    
	    In this case the del Pezzo is a cubic surface and the $(-1)$ curves are the $27$ lines on the cubic surface. There are several ways of viewing the cubic surface as the blow up of $\CP^2$ and there is a lot of classical geometry behind it. A pair of six lines satisfying some conditions forms what is called a {\em double-six}. This determines a quadric surface called a {\em Schur quadric} as well as a blow-down map to $\CP^2$ where one of the pairs of six lines are the exceptional curves. 
	       
	     The intersection of this quadric surface with the cubic surface determines a genus $4$ curve $B$ and the $K3$ surface is the double cover of the cubic surface branched at $B$. The blow down of $B$ is a nodal sextic curve  $S$ in $\CP^2$ and the $K3$ surface can also be considered as the normalization of the double cover of $\CP^2$ branched at $S$.

	     \subsection{Rational curves on del Pezzo surfaces} 
	     
	     Much of the construction does not use the fact that the curves are $(-1)$ curves but uses the fact that they are rational curves which in general meet the branch locus at double points - except at two points. The degeneracy locus is then where the two points coincide. In the case of cubics it is known by the work of Itzykson \cite{itzy} and more generally, Gottsche-Pandharipande \cite{gopa}, that given a class  ${\bf a}=(a_0,a_1,\dots,a_{9-d})$ in the Neron-Severi group of a del Pezzo of degree $d$ there exists a finitely many rational curves $C$ of class ${\bf a}$ passing through $3a_0-\sum_i a_i-1$ points for certain ${\bf a}$. The theorems are the generalization of the statement that there exists finitely many  rational curves of degree $\delta$ in $\CP^2$ passing through $3\delta-1$ points.
	     
	     The number $3a_0-\sum_i a_i$ is $\frac{1}{2}({\bf a},-2K_X)$ and so this statement can be used to say that for those ${\bf a}$ there is a rational curve meeting the branch locus $C$ in $X_d$ at $3a_0-\sum_i a_i-1$ double points and two other points. Using this rational curve in our construction along with a $(-1)$ curve which it meets one can construct a motivic cycle  which is defined outside the locus where the two points coincide - and this cycle is indecomposable. The case of $(-1)$ curves is when $3a_0-\sum a_i=1$.


	  \bibliographystyle{alpha}
	  \bibliography{AlgebraicCycles.bib}

\begin{thebibliography}{CMPMT24}

\bibitem[AN06]{alni}
Valery Alexeev and Viacheslav~V. Nikulin.
\newblock {\em Del {P}ezzo and {$K3$} surfaces}, volume~15 of {\em MSJ
  Memoirs}.
\newblock Mathematical Society of Japan, Tokyo, 2006.

\bibitem[CL05]{chle}
Xi~Chen and James~D. Lewis.
\newblock The {H}odge-{$ D$}-conjecture for {$K3$} and abelian surfaces.
\newblock {\em J. Algebraic Geom.}, 14(2):213--240, 2005.

\bibitem[CMPMT24]{CMPT}
Paola Comparin, Pedro Montero, Yulieth Prieto-Montañez, and Sergio Troncoso.
\newblock On strictly elliptic k3 surfaces and del pezzo surfaces, 2024.

\bibitem[Dol12]{dolg}
Igor~V. Dolgachev.
\newblock {\em Classical algebraic geometry}.
\newblock Cambridge University Press, Cambridge, 2012.
\newblock A modern view.

\bibitem[GP98]{gopa}
L.~G\"ottsche and R.~Pandharipande.
\newblock The quantum cohomology of blow-ups of {${\bf P}^2$} and enumerative
  geometry.
\newblock {\em J. Differential Geom.}, 48(1):61--90, 1998.

\bibitem[Itz94]{itzy}
C.~Itzykson.
\newblock Counting rational curves on rational surfaces.
\newblock In {\em Perspectives on solvable models}, pages 255--276. World Sci.
  Publ., River Edge, NJ, 1994.

\bibitem[Man91]{mani}
Yu~I. Manin.
\newblock Three-dimensional hyperbolic geometry as ...-adic arakelov geometry.
\newblock {\em Inventiones mathematicae}, 104(2):223--244, 1991.

\bibitem[MS23]{masa}
Shouhei Ma and Ken Sato.
\newblock Higher chow cycles on some k3 surfaces with involution, 2023.

\bibitem[Sat23]{sato}
Ken Sato.
\newblock A group action on higher chow cycles on a family of kummer surfaces.
\newblock 2023.

\bibitem[Sat24]{satodelpezzo}
Ken Sato.
\newblock Higher chow cycles on k3 surfaces attached to plane quartics, 2024.

\bibitem[Sha80]{shah}
Jayant Shah.
\newblock A complete moduli space for {$K3$} surfaces of degree {$2$}.
\newblock {\em Ann. of Math. (2)}, 112(3):485--510, 1980.

\bibitem[Sre14]{sree2014}
Ramesh Sreekantan.
\newblock Higher {C}how cycles on {A}belian surfaces and a non-{A}rchimedean
  analogue of the {H}odge-{$D$}-conjecture.
\newblock {\em Compos. Math.}, 150(4):691--711, 2014.

\bibitem[Sre24]{sreeK3}
Ramesh Sreekantan.
\newblock Indecomposable motivic cycles on k3 surfaces of degree 2, 2024.

\bibitem[vLW23]{wiva}
Ronald van Luijk and Rosa Winter.
\newblock Concurrent lines on del {P}ezzo surfaces of degree one.
\newblock {\em Math. Comp.}, 92(339):451--481, 2023.

\end{thebibliography}

\end{document}